\documentclass[a4paper]{article}
\usepackage{amsthm,amssymb,amsmath,enumerate,graphicx}
\usepackage{tikz}
\usepackage[british]{babel}

\theoremstyle{plain}

\newtheorem{definition}{Definition}

\newtheorem{theorem}[definition]{Theorem}
\newtheorem{corollary}[definition]{Corollary}
\newtheorem{lemma}[definition]{Lemma}

\newtheorem{claim}[definition]{Claim}
\newtheorem{noCounterRemark}[definition]{Remark}

\newcommand{\comment}[1]{}
\newcommand{\N}{\mathbb N}

\renewcommand{\subset}{\subseteq}

\newcommand{\emtext}[1]{\text{\em #1}}

\newcommand{\sm}{\setminus}

\newcommand{\ch}{\mathrm{ch}}
\newcommand{\tchi}{\chi''}

\title{List edge-colouring and total colouring in graphs of low treewidth}
\author{Henning Bruhn, Richard Lang, Maya Stein\footnote{Partially supported by Fondation Sciences Math\'ematiques de Paris}}
\date{}

\begin{document}
\maketitle

\begin{abstract}
We prove that the list chromatic index of a graph of maximum degree $\Delta$ and treewidth $\leq \sqrt{2\Delta} -3$ is $\Delta$;  and that the total chromatic number of a graph of maximum degree $\Delta$ and treewidth $\leq \Delta/3 +1$ is $\Delta +1$. This improves results by Meeks and Scott.
\end{abstract}

\section{Introduction}
  \label{sec:introduction}

We treat two common generalisations of graph colouring: list colouring and total colouring.
In analogy to the chromatic number, the \emph{list chromatic number} $\ch(G)$ is the smallest
integer $k$ so that for each choice of $k$ legal colours at every vertex, there 
is a proper colouring that picks a legal colour at every vertex. 
In a similar way, 
the \emph{list chromatic index} $\ch'(G)$ 
generalises the chromatic index.

  
While the list chromatic number and chromatic number may differ widely, the same is not true for the list chromatic index and the chromatic index. No example is known where these invariants differ. Whether  this is a general truth is
  one of the central open
  questions in the field of list colouring:
\newtheorem*{lecc}{List edge-colouring conjecture} 
 \begin{lecc}
  Equality $\ch'(G) = \chi'(G)$ holds for all graphs~$G$.\sloppy
 \end{lecc}
  The conjecture appeared for the first time in print in 1985 in~\cite{boha85}. But,
  according to Alon \cite{Alon93}, Woodall \cite{woodall01} and Jensen and Toft \cite{JeTo95}, the conjecture was suggested independently by Vizing, 
  Albertson, Collins, Erd\H os, Tucker and Gupta in the late seventies.
The most far reaching result is certainly that of Galvin~\cite{journals/jct/Galvin95}, who proved 
that $\ch'(G)=\Delta(G)$, whenever $G$ is a bipartite graph.

While list colouring generalises either vertex or edge colouring, total colouring 
applies to both, vertices and edges. The \emph{total chromatic number} $\tchi(G)$
is the smallest integer $k$ so that there is a vertex colouring of the graph $G$
with at most $k$ colours and at the same time an edge colouring with the same $k$ colours,
so that no edge receives the same colour as any of its endvertices.
 If the list edge-colouring conjecture 
is true an easy argument\footnote{If we colour the vertices of $G$ using the 
 colours $1, \ldots, \Delta(G)+3$, then for each edge there are still $\Delta(G)+1$ colours  available.
 We can colour the edges from those sets if the list edge-colouring conjecture
holds.} shows that $\tchi (G) \leq \Delta (G) +3$ for all graphs $G$.
The next conjecture asserts a little more:
\newtheorem*{tcc}{Total colouring conjecture}
 \begin{tcc}
$\tchi(G) \leq \Delta(G)+2$ holds for all graphs $G$.
 \end{tcc}
  The conjecture has been proposed independently by Behzad~\cite{behz63} and Vizing~\cite{viz-total-76} during the seventies.
  
  It is clear that  $\ch' (G)$ is bounded from below by $\Delta (G)$, the maximum degree of $G$. Also, $\tchi (G)\geq \Delta (G)+1$, since
  a vertex of maximum degree and its incident edges have to receive distinct colours.
 We show that these trivial lower bounds are already sufficient for graphs of low treewidth and high maximum degree. (The treewidth of a graph is 
   a way to measure how much the graph resembles a tree, a proper definition is given in Section~\ref{sec:treewidth}.)
 In particular, our results imply
   the list edge-colouring conjecture as well as the total colouring conjecture
for these classes of graphs. 

 \begin{theorem}
  \label{thm:quadratic-bound}
  Let $G$ be graph of treewidth~$k$ and maximum degree $\Delta(G) \geq {(k+3)^2}/{2}$. Then $\ch'(G) = \Delta(G)$.
 \end{theorem}
 
 \begin{theorem}
 \label{thm:total-bound}
 Let $G$ be a graph of treewidth  $k\geq 3$ and 
maximum degree $\Delta(G) \geq 3k-3$. Then $\tchi(G) = \Delta(G) +1$.
 \end{theorem}

  Our proofs rely on the fact that graphs with
 low treewidth and a high maximum degree contain substructures that are suitable for classical colouring arguments. This method has been used before: 
 Zhou, Nakano and Nishizeki~\cite{DBLP:journals/jal/ZhouNN96} show that $ \chi'(G)=\Delta(G)$
if the graph $G$ has treewidth $\geq 2\Delta(G)$;
Juvan, Mohar and Thomas~\cite{Juvan99listedge-colorings} prove
that the edges of any graph of treewidth~$2$ can be coloured
 from lists of size $\Delta$; and in~\cite{lang-13-1} the latter results are extended to 
graphs of treewidth~$3$ and maximum degree~$\geq 7$.
Finally, this approach has also been employed by  Meeks and Scott~\cite{journals/corr/abs-1110-4077},
who prove that  determining the list chromatic index as well as the 
list total chromatic number is fixed parameter tractable, when parameterised by treewidth. 
As a by-product they obtain that $\ch'(G)=\Delta(G)$
 and $\tchi(G)=\Delta(G)+1$ for all graphs $G$ of treewidth $k$ and maximum degree $\geq(k+2)2^{k+2}$. 
 Our main results give an improvement of their results by making the bound on the maximum degree quadratic/linear instead of exponential.
 
 The rest of the article is organised as follows. In the next section we will prove a lemma that provides a useful substructure, if applied to a graph
 of low treewidth and high maximum degree. This lemma will be used for the proofs of both our main results. The last two sections are independent of each other. In Section~\ref{sec:list-colouring} we give a proof 
 of Theorem~\ref{thm:quadratic-bound} and in Section~\ref{sec:total-colouring} we show Theorem~\ref{thm:total-bound}. We remark that if we replace the bound $\Delta(G) \geq 3k-3$ in Theorem~\ref{thm:total-bound}, with the bound $\Delta(G) \geq 3k-1$, then Theorem~\ref{thm:total-bound} becomes substantially easier to prove: all after Remark~\ref{rem:easy-total-bound} will be unnecessary.

 \section{A structural lemma}
 \label{sec:treewidth}
 
 We follow the notation of Diestel~\cite{Diestel00}.  Let us recall the definition of a tree-decomposition and of treewidth.
 For a graph $G$ a \emph{tree decomposition} $(T,\mathcal{V})$ consists of a tree $T$ and a collection $\mathcal{V} = \{V_t \textit{ : } t \in V(T) \}$ 
 of \emph{bags} $V_t \subset V(G)$ such that
 \begin{itemize}
  \item $V(G) = \bigcup_{t \in V(T)} V_t,$
  \item for each $vw \in E(G)$ there exists a $t \in V(T)$ such that $v$, $w \in V_t$ and
  \item if $v \in V_{t_1} \cap V_{t_2}$ then $v \in V_t$ for all vertices $t$ that lie on the path connecting $t_1$ and $t_2$ in $T.$
 \end{itemize}
 
 A tree decomposition $(T,\mathcal{V})$ of a graph $G$ has \emph{width} $k$ if all bags have  size at most~$k+1$. Note that in this case, if $t$ is a leaf in $T$, then the degree of the vertices in
 $V_t\setminus \bigcup_{t'\neq t}V_{t'}$ is bounded by $k$.
 The \emph{treewidth} of $G$ is the smallest number $k$ for which there exists a width $k$ tree decomposition of $G$. 
 
 Given a tree decomposition $(T,\mathcal{V})$ of $G,$ where $T$ is rooted in some vertex $r \in V (T ),$ we define the \emph{height}~$h(t)$ of
 any vertex $t \in V(T)$ to be the distance from $r$ to $t$. For  $v \in V(G)$ we define~$t_v$ as the (unique) vertex of minimum height in $T$ for which $v \in V_{t_v}$.
 In particular, if $v\in V_r$, then $t_v=r$.

 The proof of the following lemma can be extracted from~\cite{journals/corr/abs-1110-4077}. 
For the sake of completeness we include a proof here.
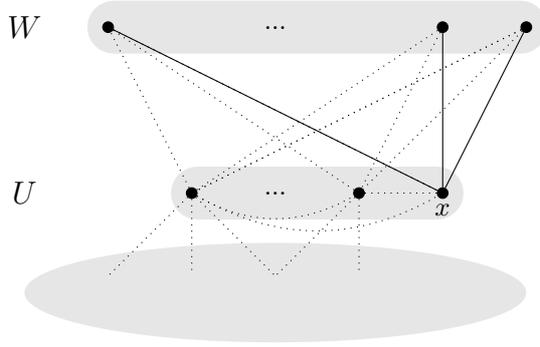
\begin{figure}
\centering
    \begin{tikzpicture}[scale=1.1]
    \tikzstyle{vertex}=[circle,draw,minimum size=4pt,inner sep=0pt,fill]
    \tikzstyle{edge-dot} = [draw,dotted]
    \tikzstyle{edge-lin} = [draw,-]
    \tikzstyle{weight} = [font=\small,draw,fill           = white,
                                  text           = black]
    \foreach \pos/\name in {{(1,0)/v_1},  {(3,0)/v_3},
                            {(0,2)/w_1}, {(4,2)/w_3},{(5,2)/w_5},{(4,0)/u}}
       \node[vertex, align=center] (\name) at \pos {};
    \foreach \pos/\name in {{(0,-1)/x_0},{(2,-1)/x_1},{(5,-1)/x_4},{(3,-1)/x_3},{(1,-1)/x_2}}
       \node[circle,inner sep = 0pt] (\name) at \pos {};
    \node [circle,inner sep = 0, minimum size = 12 pt, below] (test) at (4,0) {$x$};

    \node [circle] (dots) at (2,0) {...};
    \node [circle] (dots2) at (2,2) {...};

    \node[circle] (W) at (-1,2) {\large{$W$}};
    \node[circle] (U) at (-1,0) {\large{$U$}};

    \foreach \source/ \dest in {v_1/w_1,w_1/u/4,w_3/v_3,v_3/w_1, 
                                w_5/v_1,w_5/v_1,w_5/v_3,v_3/u,v_1/x_1,v_1/x_0,v_3/x_3,v_3/x_1,v_1/x_2,v_1/w_3}
       \path[edge-dot] (\source) -- node[] {} (\dest);

    \foreach \source/ \dest in {u/w_1, u/w_3,w_5/u}
       \path[edge-lin] (\source) -- node[] {} (\dest);

     \draw (v_1) edge[out=-30,in=-150,dotted] (u);
     \draw (v_1) edge[out=-30,in=-150,dotted] (v_3);

     \path[draw,opacity=.1,line width=20,line cap=round, color = black] (v_1) --node[] {} (u);
     \path[draw,opacity=.1,line width=20,line cap=round, color = black] (w_1) --node[] {} (w_5);
     \fill [color=black,opacity=0.1] (2,-1.2) ellipse (3 and 0.6);
      
      \end{tikzpicture}
\caption{A useful substructure.}
\label{fig:substructure}
\end{figure}
 
 \begin{lemma}[Meeks and Scott~\cite{journals/corr/abs-1110-4077}]
 \label{lem:treewidth-k-configurations}
 For $\Delta_0,$ $k \in \mathbb{N}$ with $\Delta_0 \geq 2k-1,$ let $G$ be a graph of treewidth at most $k$ 
 and $$\deg(v) + \deg(w) \geq \Delta_0+2$$ 
 for each edge $vw\in E(G)$. Then there are disjoint vertex sets $U, W \subset V(G) $ and a vertex $x \in U,$
 such that
  \begin{enumerate}[\rm (a)]
    \item $W$ is stable with $N(W) \subset U;$\label{a}
   \item $\deg(w) \leq k$ for every $w \in W;$ \label{b}
   \item $x$ is adjacent to each vertex of $W;$ and\label{c}
     \item $|U| \leq k+1$ and $|W| \geq \Delta_0 +2 -2k.$\label{d}
  \end{enumerate}
 \end{lemma}
 
 \begin{proof}
  
 By the assumptions of the lemma we have 
  \begin{equation}
 \label{equ:treewidth-k-configurations}
  \deg(v) + \deg(w)  \geq \Delta_0+2  \geq 2k+1.
 \end{equation}
In particular, of any two adjacent vertices, at least one has degree at least $k+1$ (and $G$ has at least one vertex of degree at least $k+1$).  We define ${B} \subset V (G)$ to be the (non-empty) set of vertices of degree at least $k+1$.
 Then ${S} := V(G)\sm B$ is stable. 
 
 Fix a width $k$ tree decomposition $({T},\mathcal{V})$ of $G$ and root the associated tree ${T}$ in an 
 arbitrary vertex $r \in V({T})$. 
 Let ${x} \in B$ such that 
 $
 h(t_{x} ) = \max_{v \in B} h(t_v ).
$
   Define 
 ${T'}$ as the
 subtree of ${T}$ rooted at $t_{x},$ that is, the subgraph of ${T}$ induced by all vertices $t \in V({T})$ where the path from $t$ to the root $r$ contains $t_{x}$.  
 
 Set $U:= V_{t_x}$ and $X := \bigcup_{t \in V({T'})}V_t$. Note that $|U| \leq k+1$.
 We have $B\cap X \subset U,$ since any $v \in (B \cap X)\sm  U$ would have $h(t_v)>h(t_x)$, 
contrary to the choice of ${x}$.
  Consequently 
  \begin{equation}\label{X'}
  X\sm U \subset {S}.
 \end{equation} 
  
By definition of the tree decomposition, 
 no element of $X\sm U$ can appear in a bag indexed
 by a vertex $t \in V({T} - {T'})$. Since ${S}$ is stable this gives
  \begin{equation}\label{neighX'}
  N(X\sm U) \subset U.
 \end{equation}
  By definition of $t_{x},$ also ${x}$ does not appear in any bag $V_t$ of a vertex $t \in {T}- {T'}$. So, $N(x) \subset X$.
 
    Set $W := N({x}) \sm U$. Then $W \subset X\sm U$. So by~\eqref{X'}, we can guarantee~\eqref{b}, and by~\eqref{neighX'}, we have~\eqref{a}. Also, assertion~\eqref{c} and the first part of~\eqref{d} hold.
    
    Using the assumptions of the lemma and~\eqref{b}, we get
  $$\deg({x}) \geq \Delta_0+2 - \deg(w) \geq \Delta_0 + 2 -k.$$ Since $N({x}) \subset U \cup W$ we obtain 
  $$|W| \geq |N({x}) \sm (U \sm \{{x}\})| \geq \Delta_0+2 -2k,$$
  which is as desired for the second part of~\eqref{d}.
 \end{proof}

 \section{List edge-colouring}
 \label{sec:list-colouring}

To define the list edge-colouring of a graph $G$, we define an \emph{assignment of lists} as
a function
 $L:E(G) \rightarrow \mathcal{P}(\mathbb{N})$ that maps the edges of $G$
  to \emph{lists of colours} $L(v)$. 
 A function $\gamma:E(G) \rightarrow \mathbb{N}$ is called an
 \emph{$L$-edge-colouring} of $G$, if $\gamma(e) \in L(e)$ for each $e \in E(G)$ and if no two edges 
with a common endvertex receive the same
 colour. The \emph{list chromatic index} $\ch'(G)$ is the smallest integer $k$ such that for each assignment of lists $L$ to $G$, where all lists have size $k$, there is an $L$-edge-colouring of $G$.
 
For the remainder of this section we suppose all bipartite graphs 
to have bipartition classes $U$ and $W$, unless stated otherwise. 
  
 Let $G$ be a graph with an assignment of lists $L:E(G) \rightarrow \mathcal{P}(\mathbb{N})$ to the edges of $G$. Suppose that for some stable subset $W' \subset V(G)$ we can find an $L$-edge-colouring of $G-W'$. In order
 to extend this to an $L$-edge-colouring of $G$ we have to colour the edges of the bipartite graph $H$  induced by the edges adjacent to~$W'$.
 Note that in the colouring problem we now have for $H$, 
the list of each edge $vw$ with $w \in W'$ has  size of at least $\Delta-\deg_{G-H}(v)\geq\deg_H(v)$. 

This motivates the following notion.
 For a bipartite graph $G$, we call a non-empty subset $C \subset W$ \emph{choosable}, if for any assignment of lists $L$ to the edges of the 
 induced graph $H=G[C \cup N(C)]$ with $|L(vw)| \geq d_H(v)$ for each edge $vw$ with $w\in C$ and $v \in N(C)$, there is an $L$-edge-colouring of $H$. 
  
   \begin{lemma}
\label{lem:quadratic-bound-choosable}
 Let $G$ be a (non-empty) bipartite graph with $2|W| > |U|(|U|-1)$.  Then $W$ contains a choosable subset. 
\end{lemma}
 
 To prove this we will use the following refined version of Galvin's theorem:

 \begin{theorem}[Borodin, Kostochka and Woodall~\cite{journals/jct/BorodinKW97}]
 \label{thm:Borodin-bipart-max}
  Let $G$ be a bipartite graph with an assignment of lists $L$ to the edges of $G$ such that such that 
  $|L(vw)| \geq \max\{\deg(v),\deg(w)\}$ for each edge $vw \in E(G)$. Then $G$ has an $L$-edge-colouring.
 \end{theorem} 
 
  \begin{corollary}
 \label{cor:Borodin-bipart-max}
  Let  $G$ be a bipartite graph with $\deg(v)\geq \deg(w)$  for each edge $vw \in E(G)$. Then $W$ is choosable.
 \end{corollary} 

\begin{proof}[Proof of Lemma~\ref{lem:quadratic-bound-choosable}]
 We proceed by induction on $k = |U|$. If $|U| = 1,$ then for any vertex $w\in W$ the set $\{w\}$ is choosable. 
 Given a graph $G$ that satisfies the assumptions of the lemma and for which $|U| = k+1$,
 we can assume that there is a vertex $v \in U$ of degree at most $k$. Otherwise $W$ itself is choosable by Corollary~\ref{cor:Borodin-bipart-max}: indeed, we have $\deg(w)\leq k$ for every $w\in W$ 
as $w$ has all its neighbours in $U$, which is of size~$k$.

 Let $W' := W \sm N(v)$ and $U' := U \sm \{v\}$. 
 As $|U'| = k$ and $$2|W'| = 2|W| - 2|N(v)| > (k+1)k -2k = k(k-1)$$ 
 the graph $G' = G[U' \cup W']$ fulfils the assumptions of the lemma. By the induction assumption $W'$ contains a subset of vertices that is choosable with respect to
 $G'$ and hence also choosable with respect to $G$.
\end{proof}

 \begin{proof}[Proof of Theorem~\ref{thm:quadratic-bound}]
 We prove the following assertion.
 \begin{equation*} 
\begin{minipage}[c]{0.8\textwidth}\em
Let $G$ be a graph of treewidth at most $k$ with an assignment of lists $L$ to the edges of $G,$ such that each list $L(vw)$ has size 
  $\max\{(k+3)^2,\Delta(G)\}$. Then $G$ has an $L$-edge-colouring.
\end{minipage}\ignorespacesafterend 
\end{equation*} 
Set $\Delta := \max(\frac{(k+3)^2}{2},\Delta(G))$ and let $G$ be a counterexample to the claim
  with $|V(G)|+|E(G)|$ minimal. So there are lists $L(vw)$ of size $\Delta$ for each $vw \in E(G)$, such that there is no $L$-edge-colouring of $G$.
  Clearly, $G$ is connected and non-empty. Moreover, for every edge $vw \in E(G)$ we have 
 \begin{equation*}
  \label{equ:edge-degree}
  \deg(v)+\deg(w) \geq \Delta +2.
 \end{equation*} 
  Otherwise choose an $L$-edge-colouring of $G-vw$ by minimality and observe that $L(vw)$ retains at least one available colour, which can be used to colour $vw$.
  By Lemma~\ref{lem:treewidth-k-configurations} (with $\Delta_0 = \Delta$), we know that $G$ has subsets $U, W \subset V(G),$ such that $|U| \leq k+1$ and 
  $$|W| \geq \Delta +2 -2k \geq \frac{(k+3)^2}{2} + 2 - 2k > \frac{(k+1)k}{2}.$$

  Let $H$ be the bipartite graph induced by the edges between $U$ and $W$.
  Then Lemma~\ref{lem:quadratic-bound-choosable}
  provides a subset $C \subset W$ that is choosable with respect to $H$.
  By minimality there is an $L$-edge-colouring $\gamma$ of the graph
  $G - C$. Since $C$ is choosable, we can
  extend $\gamma$ to an $L$-edge-colouring of $G$. This gives the desired contradiction.
 \end{proof}

Theorem~\ref{thm:quadratic-bound} is almost certainly not best possible. In the introduction we 
mentioned the result of Zhou et al~\cite{DBLP:journals/jal/ZhouNN96} 
that $\chi'(G)=\Delta(G)$ whenever $\Delta(G)$
is at least twice the treewidth. If one believes the list edge-colouring conjecture
then this indicates that in Theorem~\ref{thm:quadratic-bound} 
a maximum degree that is linear in $k$ is already sufficient
to guarantee the assertion.

One obvious way to improve the theorem would be to improve 
the bound on the size of $W$ in Lemma~\ref{lem:quadratic-bound-choosable}.
That bound, however, 
 is the best we can obtain by our simple use of Theorem~\ref{thm:Borodin-bipart-max} and its corollary. An illustration is given  in the following example.
  
Consider the family
of bipartite graphs $G_i,$ which is constructed as follows. Let $G_1$ be the complete bipartite graph with two vertices in partition class $U_1$, and one vertex in the other class, $W_1$.
We obtain $G_{i+1}$ from $G_i$ by adding one vertex to $U_i$, and $i$ vertices to $W_i$, thus obtaining  $U_{i+1}$ and  $W_{i+1}$. The vertices in $W_{i+1}\sm W_i$ are made adjacent to all vertices in $U_{i+1}$. 
(Thus, the vertex in $U_{i+1}\sm U_i$ is only adjacent to $W_{i+1}\sm W_i$.)

From the construction it is clear that $|W_i|=\sum_{j=1}^ij$ and $|U_i|=i+1$. So for each $\in\N$, we have
$$2|W_i| = 2\sum_{j =1}^i j = (i+1) i  = |U|(|U|-1).$$

Moreover, we can not apply Corollary~\ref{cor:Borodin-bipart-max} to any induced bipartite subgraph $H =G[C \cup N(C)]$ with $C \subset W_i$ for some $i$. To see this, let $C$ be any subset of $ W_i.$ 
Choose $\ell \leq i$  maximal such that there exists $w \in C \cap W_\ell\sm W_{\ell-1}$. By construction of $G_i,$ the vertex $w$ has degree $|U_{\ell}|=\ell +1$ in $H,$ but any neighbour of $w$ in $U_\ell\sm U_{\ell-1}$
 has  degree   $|W_\ell\sm W_{\ell-1}|=\ell$ in $H$, by the maximality of $\ell$. Thus Corollary~\ref{cor:Borodin-bipart-max} does not apply to $(C, N(C))$. 
 
 However, there is another
version of Galvin's theorem, which can be used to show that for any $i \geq 3,$ the set $W_i$ itself is choosable in $G_i$:
\begin{theorem}[Borodin, Kostochka and Woodall~\cite{journals/jct/BorodinKW97}]
\label{thm:slivnik}
 Let $G$ be a bipartite graph. Then $W$ is choosable if and only if $G$ has an $L$-edge-colouring from the lists $L^*(uw) = \{1, \ldots, \deg(u)\}$ for $u \in U$.
\end{theorem}

Let us show by induction that the graphs $G_i$ are colourable from the lists~$L^*$, for $i\geq 3$.
It is not hard to see that the graph $G_3$ (which equals $K_{3,3}-e$) can be coloured from the lists $L^*$. For the graph
$G_{i+1},$ consider the lists $L^*$ as in the above theorem. By induction, colour the edges of $G_i$ from the smaller lists, and colour the edges adjacent to $U_{i+1}\sm U_i$ with $1, \ldots, i$. The remaining edges lie between $W_{i+1}\sm W_i$ and $U_i$, spanning a complete bipartite $(i+1)$-regular graph $H$. Their lists retain a set $C_{i+1}$ of $i+1$  colours that are unused  so far. So we may  apply Corollary~\ref{cor:Borodin-bipart-max} to see that $W_{i+1}\sm W_i$ 
is choosable in $H$. Thus by Theorem~\ref{thm:slivnik}, we can colour the $E(H)$ with $i+1$ colours. Substitute these colours with the ones from $C_{i+1}$, and we are done.

This suggests that the bound on the size of $|W|$ in Lemma~\ref{lem:quadratic-bound-choosable}
might not be optimal. 
Perhaps Theorem~\ref{thm:slivnik} could be used in general to decrease the bound on the maximum degree.

\section{Total colouring}
\label{sec:total-colouring}
The whole section is devoted to the proof of Theorem~\ref{thm:total-bound}. The same theorem with the slightly stronger bound $\Delta(G)\geq 3k-1$ can be shown with less effort: the reader interested in this variant may read our proof up to
Remark~\ref{rem:easy-total-bound} and skip everything afterwards.

We show the following assertion, which clearly implies Theorem~\ref{thm:total-bound}:  
\[
\emtext{$\chi''(G)\leq \max\{\Delta(G),3k-3,2k\}+1$ for any graph $G$
of treewidth $\leq k$.}
\]
 Suppose this is not true, and let $G$ be an edge-minimal counterexample. 
Put $\Delta:=\max\{\Delta(G),3k-3, 2k\}$. 
 (Thus we assume $G$ cannot be totally coloured with $\Delta+1$
colours, but $G-e$ can, for any edge $e$.)  

\begin{claim}\label{degreesum}
We have $\deg(u)+\deg(v)\geq\Delta+1$ for each edge $uv\in E(G)$.
\end{claim}
\begin{proof}
Suppose $G$ contains an edge $uv$ for which the degree sum is at most $\Delta$,
where we assume that $\deg(u)\geq \deg(v)$. Let 
 $G-uv$ be totally coloured with at most $\Delta+1$ colours. 

Now, if $u$ and $v$
receive the same colour, we recolour $v$: Note that
 $v$ has $\deg(v)$ coloured neighbours and is incident with $\deg(v)-1$
coloured edges. As $$2\deg(v)-1\leq \deg(u)+\deg(v)-1\leq \Delta-1,$$ there is a colour
among the $\Delta+1$ colours available that can be given to $v$. 

Finally, we observe that 
the edge $uv$ is incident with two coloured vertices and adjacent to 
$\deg(u)+\deg(v)-2$ coloured edges. That means there are at most $\deg(u)+\deg(v)\leq\Delta$
different colours that cannot be chosen for $uv$ -- but we have $\Delta+1$ colours at our disposal.
Thus, $G$ can be totally coloured with $\Delta+1$ colours.
\end{proof}

\medskip
By Claim~\ref{degreesum} we may apply Lemma~\ref{lem:treewidth-k-configurations} with parameters $\Delta_0 = \Delta -1$ and $k$; let $U,W,x$ as obtained by the lemma.
We choose a neighbour $w^*\in W$ of $x$ and totally colour $G-w^*x$  with at most~$\Delta+1$
colours. Further, we uncolour every vertex in $W$. 
Observe that it will not be a problem to colour $W$ once all the rest of $V(G)\cup E(G)$ has been coloured: The vertices in $W$ have degree at most~$k$ each, so there will be at most $2k\leq \Delta$ 
forbidden colours at each $w\in W$. 

We will say that a colour $\gamma$ is \emph{missing at a vertex $v$},
if neither $v$ nor any incident edge is coloured with $\gamma$ (neighbours of $v$, though, are allowed
to have colour~$\gamma$). Let $M(v)$ be the set of all colours missing at $v$. 

As $x$ is incident with at most $\Delta-1$ coloured edges, there is a colour $\alpha$ missing at $x$.
Call an edge coloured $\alpha$ an \emph{$\alpha$-edge}. Note that 
\begin{equation}\label{alphanotmissing}
 \alpha\notin M(w^*).
 \end{equation}
  Indeed,  otherwise we could 
colour $w^*x$ with~$\alpha$, then colour $W$ as described above, and thus get a $(\Delta +1)$-colouring of $G$, which by assumption does not exist. 

Let $F$ be the set of colours on edges between $x$ and $U$
together with the colour of $x$ itself. Note that, since $|U|\leq k+1$, we have that
\begin{equation}\label{sizeF}
|F|\leq k+1.
\end{equation}

Colours that are not in $F$, but missing at $w^*$ are  useful to us, because they could be used 
to colour $xw^*$ (after possibly recolouring some edges in $E(U,W)$). Let us make this more precise:

\begin{claim}\label{alphamatching}
For every colour $\beta\in M(w^*)\sm F$ there is a vertex $v_\beta \in W$ so that $xv_\beta$ has colour $\beta$. Furthermore, there is an $\alpha$-edge incident with $ v_\beta$.
\end{claim} 
\begin{proof}
If
there is no $v_\beta\in W$ with $xv_\beta$  
coloured  $\beta$, then, since $\beta\notin F$,
the colour $\beta$ is also missing at $x$, and we may use it for the edge $xw^*$. This proves the first part of the claim.

Next,  
if $\alpha$ is missing at $v_\beta$, we can colour $xv_\beta$ with $\alpha$ and $xw^*$ with $\beta$. Colouring $W$ as described above, this
 gives a $(\Delta +1)$-colouring of $G$, a contradiction. 
Thus, we may assume that $\alpha$ is not missing at $v_\beta$, which, as the vertices of $W$ are uncoloured,
means that there is an $\alpha$-edge at $v_\beta$. 
\end{proof}

Denote by $n_\alpha$ the number of $\alpha$-edges between $U$ and $W$. Using Claim~\ref{alphamatching} and
the fact that there is an $\alpha$-edge at $w^*$ by~\eqref{alphanotmissing}, we see that
\begin{equation}\label{howmanyalphaedges}
n_\alpha\geq | M(w^*)\sm F|+1.
\end{equation}

Let us now estimate how many colours are missing at $w^*$. Of the $\Delta+1$ colours available,
at most $\deg(w^*)-1\leq k-1$ are used for incident edges, and none on $w^*$.

Thus, 
\begin{equation}\label{elvis}
|M(w^*)|\geq\Delta +1 - (\deg(w^*) -1) \geq 2k-1.
\end{equation}

\begin{noCounterRemark}
\label{rem:easy-total-bound}
 Our argumentation so far is enough to prove that any graph of treewidth $k$ and maximum degree $\Delta(G) \geq 3k-1$ satisfies $\tchi(G) = \Delta(G)+1$.\\ Indeed, note that with the assumption $\Delta(G) \geq 3k-1$, we obtain $|M(w^*)|\geq 2k+1$ in~\eqref{elvis}. Plugging this into~\eqref{howmanyalphaedges}, and using~\eqref{sizeF}, we get $n_\alpha \geq k+1$. On the other hand, the $\alpha$-edges form a matching, 
which means there can be at most $k$, as $\alpha$ is missing at $x$.
\end{noCounterRemark}

Let $\rho_x$ be the colour of $x$. 

\begin{claim}\label{FsubsetM}
We have $ F - \rho_x  \subset M(w^*)$. Moreover, $\rho_x\in M(w^*)$ if and only if there is a vertex in $U$ that is coloured $\alpha$.
\end{claim}
\begin{proof}
Let $u_\alpha$ be the number of vertices of $U$ coloured $\alpha$. 
No vertex in $U$ may be incident with two of the $\alpha$-edges counted by $n_\alpha$. 
As, moreover,
$\alpha$ is missing at $x$, we get that
\begin{equation}\label{aaa}
n_\alpha \leq |U|-u_\alpha -1 \leq k-u_\alpha .
\end{equation}
On the other hand,
\begin{equation}\label{bbb}
 |M(w^*)\sm F|- |F\sm M(w^*)|=|M(w^*)|-|F|\overset{\eqref{sizeF}, \eqref{elvis}}\geq   k-2.
\end{equation}
Putting~\eqref{howmanyalphaedges}, \eqref{aaa} and \eqref{bbb} together, we get
\[
k- u_\alpha \geq  |F\sm M(w^*)|+ k-1.
\]
In other words, $$ |F\sm M(w^*)|+u_\alpha \leq 1.$$

In the case $u_\alpha >0$, this proves the claim. So suppose $u_\alpha =0$.
If $\rho_x \in M(w^*)$, we can 
recolour $x$ with $\alpha$, colour the edge
$xw^*$ with $\rho_x$ and colour $W$ as above. Therefore, $\rho_x \notin M(w^*)$, and the claim follows.
 \end{proof}

\begin{claim}\label{sizeFdegw*}
We have $ |F |=k+1$ and $\deg(w^*)=k$. 
\end{claim}
\begin{proof}
Suppose either of the two inequalities does not hold. Then the estimate in~\eqref{bbb} 
is never tight, and we deduce
\begin{equation*}
 |M(w^*)\sm F|- |F\sm M(w^*)|\geq   k-1.
\end{equation*}
This leads to $$ |F\sm M(w^*)|+u_\alpha \leq 0.$$
Thus both $u_\alpha =0$ and $\rho_x \in M(w^*)$, contradicting Claim~\ref{FsubsetM}.
 \end{proof}

We next investigate which colours are missing at the vertices $v_\beta$ from Claim~\ref{alphamatching}.
\begin{claim}\label{tinhattrio}
$M(v_\beta) \subset  M(w^*)$ for every colour $\beta\in M(w^*)\sm F$.
\end{claim} 
\begin{proof}
First, note that $\rho_x \notin M(v_\beta)\sm M(w^*)$. Indeed, otherwise $\rho_x \notin M(w^*)$ and
therefore, by Claim~\ref{FsubsetM}, no vertex in $U$ is coloured with $\alpha$. Thus we can recolour $xv_\beta$ with $\rho_x,$ colour $xw^*$ with $\beta,$ recolour $x$ with $\alpha$
and finish by colouring $W$.

 Now, for contradiction suppose there is a colour $\beta^*\in M(v_\beta)\sm M(w^*).$ By the previous paragraph, $\beta^* \neq \rho_x$. Hence, by Claim~\ref{FsubsetM}, $\beta^* \notin F$.
 
 Then, there must 
be a vertex $y\in W$ so that $xy$ has colour $\beta^*$, as
otherwise we can colour the edge $xw^*$ with colour $\beta$, and the edge $xv_\beta$ with colour $\beta^*$, colour~$W$, and are done.  
Moreover, $y$ 
is incident with an  $\alpha$-edge. Indeed, otherwise we can colour the edge $xy$ with $\alpha$, the edge $xw^*$ with $\beta$, and the edge $xv_\beta$ with $\beta^*$, colour~$W$, and are done.  

Setting $\delta=1$ if $\rho_x\in M(w^*)$ and $\delta=0$ otherwise, we deduce from 
Claim~\ref{alphamatching} and~\eqref{alphanotmissing} that
\begin{equation*}
n_\alpha + \delta\geq | M(w^*)\sm (F\sm\{\rho_x\})|+2\overset{\eqref{sizeF}, \eqref{elvis}}\geq   k+1 .
\end{equation*}

On the other hand, using the second part of Claim~\ref{FsubsetM}, we see that
$$ n_\alpha + \delta\leq |U|-1\leq k,$$
a contradiction.
\end{proof}

Fix $\beta\in M(w^*)\sm F$. Note that $|M(v_\beta)| \geq |M(w^*)|-1$, as $\deg(v_\beta)\leq k=\deg(w^*)$ by Claim~\ref{sizeFdegw*}. So, by Claim~\ref{tinhattrio}, we get that $M(v_\beta) = M(w^*) \sm\{ \beta\}$. In particular,  $F-\rho_x\subseteq M(v_\beta)$.

By Claim~\ref{alphamatching}, there is a vertex $u\in U$ be so that $v_\beta u$ has colour $\alpha$. The edge $ux$
exists as  $|F|=k+1$  by Claim~\ref{sizeFdegw*}. The colour $\rho_{ux}$ of $ux$ is in $F$, and thus missing at $v_\beta$. So we may
swap  colours on $ux$ and $uv_\beta$. This yields again a
total colouring of $(E-xw^*)\cup V\sm W$. In the new colouring $\rho_{ux}$ is missing at $x$.
As $\rho_{ux}$ is also missing at $w^*$ we may use it to colour $xw^*$. 
Finally we fix the colours of the vertices in $W$ in order to obtain 
a total colouring of $G$.

\bibliographystyle{amsplain}
\bibliography{mylib}

\small
\vfill
\noindent
Version 11 Nov 2013
\bigbreak

\noindent
\begin{tabular}{cc}
\begin{minipage}[t]{0.5\linewidth}
Henning Bruhn\\{\tt <henning.bruhn@uni-ulm.de>}\\
Universit\"at Ulm, Germany
\end{minipage}
&
\begin{minipage}[t]{0.5\linewidth}
Richard Lang\\ 
{\tt <rlang@dim.uchile.cl>}

Maya Stein\\
{\tt <mstein@dim.uchile.cl>}\\
Universidad de Chile, Chile 
\end{minipage}
\end{tabular}

\end{document}